\documentclass[a4paper,11pt]{article}
\usepackage{amsmath,amssymb,amsthm, multirow,tabularx}
\usepackage[english]{babel}
\usepackage[all,cmtip]{xy}
\usepackage[latin1]{inputenc}

\newcommand{\C}{\mathbb{C}}
\newcommand{\R}{\mathbb{R}}
\newcommand{\Z}{\mathbb{Z}}
\newcommand{\N}{\mathbb{N}}

\newcommand{\e}{\mathfrak{e}}
\newcommand{\f}{\mathfrak{f}}
\newcommand{\g}{\mathfrak{g}}
\newcommand{\p}{\mathfrak{p}}
\newcommand{\n}{\mathfrak{n}}
\newcommand{\s}{\mathfrak{s}}

\newcommand{\m}{\mathfrak{m}}
\renewcommand{\a}{\mathfrak{a}}
\renewcommand{\k}{\mathfrak{k}}

\renewcommand{\d}{\mathfrak{d}}

\newcommand{\ad}{{\rm ad}}
\newcommand{\adj}{{\rm Ad}}
\renewcommand{\L}{{\rm Lie}}
\newcommand{\Aut}{{\rm Aut}}
\newcommand{\Iso}{{\rm Iso}}

\renewcommand{\O}{\mathcal{O}}

\def\qq{/\kern-.185em /}

\makeatletter
\def\hlinewd#1{%
\noalign{\ifnum0=`}\fi\hrule \@height #1 %
\futurelet\reserved@a\@xhline}
\makeatother

\title{Quotients of the crown domain by a proper action of a cyclic group.}
\author{Sara Vitali}

\begin{document}

\pagestyle{empty}

\theoremstyle{plain}
\newtheorem{teo}{Theorem}[section]
\newtheorem*{thm*}{Theorem}
\newtheorem*{pro*}{Proposition}
\newtheorem{lemma}[teo]{Lemma}
\newtheorem{prop}[teo]{Proposition}
\newtheorem{cor}[teo]{Corollary}
\newtheorem{claim}[teo]{Claim}

\theoremstyle{definition}
\newtheorem{defi}[teo]{Definition}
\newtheorem{rmk}[teo]{Remark}

\theoremstyle{remark}
\newtheorem{ex}[teo]{{\rm EXAMPLE}}

\maketitle

\pagestyle{plain}

\begin{abstract} Let $G/K$ be an irreducible Riemannian symmetric space of the non-compact type and denote by $\Xi$ the associated crown domain. We show that for any proper action of a cyclic group $\Gamma$ the quotient $\Xi/\Gamma$ is Stein. An analogous statement holds true for discrete nilpotent subgroups of a maximal split-solvable subgroup of $G$. We also show that $\Xi$ is taut.
\end{abstract}

\section*{Introduction}

\

Let $G/K$ be an irreducible Riemannian symmetric space of the non-compact type, where $G$ is assumed to be embedded in its universal complexification
$G^\C$. In \cite{AkGi90} D. N. Akhiezer and S. G. Gindikin
pointed out a distinguished  invariant domain $\Xi$  of $G^\C/K^\C$
containing  $G/K$ (as a maximal  totally-real submanifold)  such that
the extended (left) $G$-action on $\Xi$ is proper.
The domain $\Xi$, which is usually referred to as the crown domain or the Akhiezer-Gindikin domain, is Stein and Kobayashi hyperbolic by a result
of D. Burns, S. Hind and S. Halverscheid (\cite{BHH03}, cf. \cite{Bar03},  \cite{KrSt05}). In fact G. Fels and A.T. Huckleberry (\cite{FeHu05}) have shown that
with respect to these properties it is the maximal
$G$-invariant complexification of $G/K$ in $G^\C/K^\C$.
By using the characterization of the $G$-invariant, plurisubharmonic functions on $\Xi$
given in \cite{BHH03}, here we also note that $\Xi$ is taut (Prop.~\ref{TAUT}).
It seems not to be known whether all crown domains are complete
Kobayashi hyperbolic.

The crown domain
can also be regarded as the maximal domain
in the tangent bundle of $G/K$ admitting an
adapted complex structure (see \cite{BHH03} for more details).
Recently it has been intensively investigated in
connection with the harmonic analysis of $G/K$
(see, e.g. \cite{GiKr02b}, \cite{KrSt04}, \cite{KrSt05},
\cite{KrOp08}).
Here we consider
particular complex manifolds associated to $\Xi$. Namely,
given a proper action on $\Xi$ of a discrete group $\Gamma$ of  biholomorphisms,  we
are interested in complex-geometric properties of the
quotient $\Xi/\Gamma$. If $\Gamma$ is finite  one knows
that  $\Xi/\Gamma$ is Stein by a classical result of H. Grauert and R. Remmert (\cite[Thm.~1,Ch.~V]{GrRe79}, cf.~\cite{Hei91}).

So the simplest interesting case is that of
an infinite cyclic group  $\Gamma \cong \Z$.
One should observe that $\Xi$ is biholomorphic to a
simply connected domain in $\C^n$.
For this, consider an Iwasawa decomposition $N\!AK$ of $G$.
Then $\Xi$ can be realized as
an $N\! A$-invariant domain in the universal complexification
of $N\! A$, which is
biholomorphic to a complex affine space (cf. Sect. 5).
In this setting one may ask the
following  more general question: given proper $\Z$-action on
a simply connected Stein domain $X$ of $\C^n$, is the quotient
 $X/\Z$ Stein ?
 This is not always the case; for instance
J. Winkelmann (\cite{Win90})
has given an example of a free and properly discontinuous action on
$\C^5$ such that the quotient is not holomorphically separable.
On the other hand if $X$ is a simply-connected, bounded, Stein domain of $\C^2$ and the $\Z$-action is induced by a (proper) $\R$-action,
 then $X/\Z$ is Stein by a result of C. Miebach and
K. Oeljeklaus (\cite{MiOe09}). Furthermore,  C. Miebach (\cite{Mie10})
has shown that if $X$ is a
homogeneous bounded
domain, then $X/\Z$ is Stein for any proper
$\Z$-action.
For $n \ge 2$
we are not aware of any example of a simply connected,
bounded, Stein domain of $\C^n$, with a proper $\Z$-action such that
the quotient is not Stein.

The crown domain is either a Hermitian symmetric space of a larger group or it is rigid, i.e. the automorphism group of $\Xi$ coincides with the group of isometries of $G/K$
(\cite{BHH03}). Hence, the latter case is
a source of interesting examples of simply connected,
non-homogenous,  Stein
domains of $\C^n$ with a large authomorphism group.
Our main result is

\begin{thm*}
Let $G/K$ be an irreducible Riemannian symmetric space of the
non-compact type and let $\Xi$ be the associated crown domain.
Then $\Xi/\Z$ is a Stein manifold  for every proper $\Z$-action.
\end{thm*}

An important ingredient in the proof is the above mentioned realization of $\Xi$ as an $N\!A$-invariant domain in the universal complexification
of $N\!A$. For this we recall that the crown domain is given
by  $\Xi= G \exp(i\omega)K^\C/K^\C$,
where  $\omega$ is the cell in the Lie algebra $\a$ of $A$
defined by
$$\omega:=\big\{X\in\a:\big|\alpha(X)\big|< \pi / 2, \;\text{ for every restricted root } \alpha \big \}.$$
By a result of S. Gindikin and B. Krötz  (\cite[Thm.~1.3]{GiKr02a}), the crown
$\Xi$ is contained in $N^\C\!A \exp(i\omega)K^\C/K^\C$. Thus, in order to realize $\Xi$ as an $N\!A$-invariant domain in the universal complexification of $N\!A$, it is sufficient to note that
the multiplication map $N^\C \times A \exp(i\omega) \to
G^\C/K^\C$, given by $(n,a) \to naK^\C$, is an open embedding
(Prop.~\ref{MULT}).
Then by following a strategy carried out in \cite{Mie10},
 one can reduce  to the case of $\Z$ contained in
 $N\!A$. Since $N\!A$ is a connected, split-solvable Lie group,
the quotient of the universal complexification of $N\! A$ by $\Z$ is Stein (Prop.~\ref{S/Z QUOTIENT}). Finally,
 $\Xi/\Z$ is locally Stein in such a Stein quotient
 and the proof of the above theorem follows by applying
a classical result of F. Docquier and H. Grauert (\cite{DoGr60}).

As pointed out to us by C. Miebach,
the above arguments also apply to show the following proposition
(Prop.~\ref{SUGG})

\begin{pro*}
Let $\Gamma$ be any discrete, nilpotent subgroup of $N\!A$.
Then $\Xi/\Gamma$ is Stein.
\end{pro*}


The paper is organized as follows. In the first section we recall basic results on semisimple Lie groups and on the Iwasawa decomposition.
In section $2$  we discuss complex-geometric properties of
quotients of complex Lie groups by a proper action of a
discrete subgroup. In sections $3$ and $4$ we recall basic properties of the crown domain and we show that $\Xi$ is taut. The main result is proved in section $5$.

\bigskip
\noindent
{\bf Acknowledgments.} The results in the manuscript  are part
of my Ph.D. thesis. I would like to express my deepest gratitude to my supervisor, Dr. Andrea Iannuzzi, for his precious guidance. I would like to thank Prof. Laura Geatti, Prof. Christian Miebach and Prof. Stefano Trapani for their helpful comments. I also wish to express my gratitude to Prof. Dmitry N. Akhiezer for suggesting the topic of the present article during his visit to our department in November 2009.


\section{Preliminaries.}\label{PRELIMINARIES}
\noindent

Here we introduce the notation and we recall some basic facts on
Riemannian symmetric spaces and semisimple Lie groups.
\begin{defi} Let $G$ be a real Lie group. A complex Lie group $G^\C$
together with a Lie group homomorphism $\gamma:G\to G^\C$ is the
\emph{universal complexification} of $G$ if it satisfies the following universal
property: for every complex Lie group $H$ and every Lie group
homomorphism $\phi:G\to H$ there exists a unique morphism
$\phi^\C:G^\C\to H$ such that $\phi=\phi^\C\circ\gamma$.
\end{defi}
By the above universal property $G^\C$ is unique up to isomorphisms.
For the existence and the construction of the universal complexification
and its fundamental properties we refer to \cite[XVII.5]{Hoc65}.

Let $G$ be a connected, non-compact, simple\footnote{Here a Lie
group $G$ is simple if its Lie algebra is simple. Therefore $G$
may have non-trivial discrete center.} Lie group which is
assumed to be embedded in its universal complexification $G^\C$.
Choose a maximal compact subgroup $K$  of $G$ and note that $K$ is
connected. The quotient space $M=G/K$ is an irreducible,
Riemannian symmetric space of the non-compact type.
The universal complexification $K^\C$ of $K$ coincides with
the complexification of $K$ in $G^\C$, i.e. with the connected Lie subgroup
of $G^\C$ associated to the complexification of the Lie algebra of $K$.
In the sequel we will be interested in the $G$-orbit structure of $G^\C/K^\C$.  One has (cf. \cite[Rem.~4.1]{GeIa08})

\begin{rmk}
\label{KLEIN}
Let $G/K$ and $G'/K'$ be two different Klein
representations of the same irreducible, Riemannian symmetric
space of the non-compact type $M$, with $\L(G)=\L(G')$ a simple Lie algebra. Then the complexifications
$G^\C/K^\C$ and $(G')^\C/(K')^\C$ are biholomorphic and they have
the same orbit structure.
\end{rmk}

Let $\k$ be the Lie algebra of $K$ and consider the Cartan
decomposition $\k\oplus\p$ of $\g$ associated to $\k$. For $\a$  a
maximal abelian subalgebra of $\p$, consider the corresponding
restricted root system
 $\Sigma\subset\a\!\setminus\!\{0\}$ and the root space decomposition
$$\g=\bigoplus_{\alpha\in\Sigma}\g^\alpha\oplus\a\oplus\m\;,$$
where $\m$ is the  centralizer of $\a$ in $\k$ and the root spaces
are defined by $\g^\alpha=\{X\in\g:[H,X]=\alpha(H)X,\; {\rm for\
every\ } H\in\a\}$. As a consequence of the Jacobi identity  one
has
$$[\g^\alpha,\g^\beta]\subset\g^{\alpha+\beta} \,.$$
Fix a system of positive roots $\Sigma^+\subset\Sigma$. One has  the
associated Iwasawa decomposition at the Lie algebra level
$$\g=\n\oplus\a\oplus\k,$$
where $\n=\bigoplus_{\alpha\in\Sigma^+} \g^\alpha$.\\
By construction $\a$ normalizes $\n$, therefore $\n\oplus\a$ is a
semidirect product of a nilpotent and an abelian algebra. In
particular $\n\oplus\a$ is a solvable Lie algebra.
Let $A$ and $N$ be the analytic subgroups of $G$ corresponding to
$\a$ and $\n$, respectively. The Iwasawa decomposition at the
group level is given by  $G=NAK$, meaning that the multiplication
map $N \times A \times K \to G$ is an analytic diffeomorphism (see
\cite[Thm.~5.1, Ch.~VI]{Hel01}). Moreover, $N\!A$ is a (closed)
solvable subgroup of $G$ isomorphic to the semidirect product $N \rtimes
A$. The following facts are well known and scattered in the
literature. For a proof see e.g. \cite[Prop.~1.3, Rem.~1.4]{Vit12}.

\begin{prop}\label{N} Let $G$ be a connected, real, simple Lie group
and let $NAK$ be an Iwasawa decomposition of $G$. Then
\begin{itemize}
    \item[{\rm(i)}] The complexification\footnote{Note that the complexification $A^\C$ of $A$ in $G^\C$ is not the universal complexification of $A$.} $A^\C$ of $A$ in $G^\C$ is
    closed and is isomorphic to $(\C^*)^r$, with $r=\dim_\R \a$.
    \item[{\rm(ii)}] The complexification  $N^\C$ of $N$ in
    $G^\C$  is closed and simply connected.
    \item[{\rm(iii)}] The group $A^\C$ normalizes $N^\C$ and
    $N^\C\cap A^\C=\{e\}$. In particular $N^\C\!A^\C$ is isomorphic to a semidirect product $N^\C\rtimes A^\C$.
    \item[{\rm (iv)}] The map $N^\C\times A^\C\times N^\C\longrightarrow G^\C$ given by $(n,a,n')\mapsto na\theta(n')$ is injective, where
  $\theta$ denotes the holomorphic extension to $G^\C$ of the Cartan involution of $G$ with respect to $K$.
\end{itemize}
\end{prop}

We also  recall some general facts regarding the
multiplicative Jordan-Chevalley decomposition. Let $G$ be a real, simple Lie group.
An element $g\in G$ is said to be
\emph{unipotent} (resp.  \emph{hyperbolic}) if it is of the form
$g=\exp(X)$, where $\ad{(X)}$ is nilpotent (resp. $\ad{(X)}$ is
diagonalizable over $\R$) and $\emph{elliptic}$ if $\adj(g)$ is
diagonalizable over $\C$ with eigenvalues of norm $1$.
Then one has
\begin{prop}\label{JORDAN}{\rm (see \cite[Prop.~2.1]{Kos73})}
Let $G$ be a real simple Lie group. Then every element $g$ of $G$
admits a unique decomposition $g_u g_h g_e $, where $g_u$ is
unipotent, $g_h$ is hyperbolic, $g_e$ is elliptic and every
pair of elements in $\{g_u,g_h,g_e\}$ commute.
\end{prop}

It is easy to check that  if $nak$ is an Iwasawa decomposition of an element $g$ of $G$, then the elements $n$, $a$ and $k$ are unipotent, hyperbolic and elliptic, respectively.
The relation between the two decompositions is clarified by the following proposition (cf. \cite[Prop.~2.3-2.5]{Kos73})
\begin{prop}\label{jc} Let $G$ be a real, non-compact, simple Lie group and let $N\!AK$ be an Iwasawa decomposition of $G$. Then an
element of $G$
\begin{description}
    \item[{\rm(i)}] is unipotent if and only if it is conjugate to an element
    of $N$,
    \item[{\rm(ii)}] is hyperbolic if and only if it is conjugate to an element
    of $A$,
    \item[{\rm(iii)}] is elliptic if and only if it is conjugate to an element of
    $K$,
    \item[{\rm(iv)}] has trivial elliptic part if and only if it is conjugate to an element of $N\!A$.
\end{description}
\end{prop}


\section{Discrete group actions on complex Lie groups.}
\label{G/H STEIN}
\noindent

Let $G$ be a connected non-compact real Lie group embedded in its universal complexification $G^\C$ and let $\Gamma$ be a discrete subgroup of $G$. Then $\Gamma$ acts freely and properly discontinuously on $G^\C$
and $G^\C/\Gamma$ is a complex manifold. We recall that the universal complexification $G^\C$ of a real Lie group is Stein (see \cite[p.~147]{Hei93}). It is of interest to know when the quotient $G^\C/\Gamma$ is Stein in terms of sufficient and/or necessary conditions on $G$ and $\Gamma$.
For instance if $G$ is nilpotent, by a result of B. Gilligan and A. T. Huckleberry (see the proof of Thm 7 in \cite{GiHu78}), it follows that $G^\C/\Gamma$ is Stein.
For Lie groups with simply connected complexification one has the following result of J. J. Loeb.

\begin{teo}{\rm \cite[Thm.~1,Lemma~1]{Loe85}}\label{LOEB} Let $G$ be a
real connected Lie group with simply connected universal complexification $G^\C$ and let $\Gamma$ be a discrete, cocompact subgroup of $G$. Then $G^\C/\Gamma$ is Stein if and only if $G$
has purely imaginary spectrum, i.e. for every $X$ in $\g$ the
eigenvalues of  $\ad(X)$ are purely imaginary.
\end{teo}

In the sequel we will be interested in the universal
complexification of a solvable Lie group. We point out that the solvability of the group $G$ is not sufficient in order to satisfy Loeb's condition of Theorem~\ref{LOEB}. Indeed one can give an example (see \cite[p.~74]{Loe85}) of a solvable Lie group $G$  with simply connected complexification $G^\C$,
admitting a discrete cocompact subgroup $\Gamma$
and such that  $\ad(X)$ has
eigenvalues with non-trivial real part, for some $X$ in $\g$.  Thus Theorem~\ref{LOEB} implies that $G^\C/\Gamma$ is not Stein. In the sequel we will be interested in the following class of solvable Lie groups.

\begin{defi} A real Lie algebra $\s$ is \emph{split-solvable}
if it is solvable and the eigenvalues of $\ad(X)$ are real for every $X\in \s$. A real Lie group $S$ is \emph{split-solvable} if it is simply connected and its Lie algebra $\s$ is split-solvable.
\end{defi}

\begin{rmk}\label{NA} Let $G$ be a real simple Lie group and let $G=N\!AK$ be an Iwasawa decomposition of $G$. Then $N\!A$ is a maximal split-solvable subgroup of $G$. Indeed, it is easy to check that $N\!A$ is simply connected and maximal solvable. Moreover one can choose
a suitable basis of $\g$ such  that $\ad(X)$
is represented by upper triangular matrices for every $X$ in  $\n\oplus\a$
(cf. \cite[Prop. 1.3]{Vit12}.
Thus $N\!A$ is a maximal split-solvable subgroup of $G$.
\end{rmk}

\begin{rmk}\label{SPLIT}
If $S$ is a split-solvable Lie group then the exponential map is a diffeomorphism (see, e.g. \cite[Thm.~6.4, Ch.~2]{Vin94}). In
particular, every connected subgroup of $S$ is closed and
simply-connected. In fact, the latter property holds true for every
simply connected solvable Lie group (cf.
\cite[Thm.~3.18.12]{Var84}).
\end{rmk}

\begin{defi} A discrete subgroup $\Gamma$ of a real split-solvable Lie group $S$ is \emph{nilpotent} if equivalently
\begin{description}
    \item[{\rm(i)}] it is contained in a connected, nilpotent subgroup of $S$;
    \item[{\rm(ii)}] it admits a finite central series $\Gamma\rhd\Gamma^{(1)}\rhd\ldots\rhd\Gamma^{(m)}
        =\{e\}$, where $\Gamma^{(1)}:=[\Gamma,\Gamma]$ and $\Gamma^{(i)}:=[\Gamma,\Gamma^{(i-1)}]$.
\end{description}
\end{defi}

\begin{prop} \label{S/Z QUOTIENT}
Let $\Gamma$ be a discrete subgroup of a connected, split-solvable Lie group $S$ and let $S^\C$ be the universal complexification of $S$. Then $S^\C/\Gamma$ is Stein if and only if $\Gamma$ is nilpotent.
\end{prop}

\begin{proof}
Since $S$ is split-solvable, by \cite[Cor.~3.4]{Wit02} there exists a unique connected subgroup $S_\Gamma$ of $S$ such that $S_\Gamma/\Gamma$ is compact. The subgroup $S_\Gamma$ being unique, it coincides with the connected subgroup associated to the Lie subalgebra
$\s_\Gamma$ of $\s:=\L(S)$ generated by $\exp^{-1}(\Gamma)$. In particular $S_\Gamma$ is the smallest connected (closed) Lie subgroup of $S$ containing $\Gamma$.
Note that since $S$ is simply connected, so is its universal complexification $S^\C$.
As a consequence the connected Lie subgroup $S^\C_\Gamma$ of $S^\C$ associated to the complexified Lie algebra $\s^\C_\Gamma$ is closed and simply connected (see \cite[Thm.~3.18.12]{Var84}). Moreover, by \cite[Thm.~1]{HuOe81}  the quotient $S^\C/S^\C_\Gamma$ is biholomorphic to $\C^k$.
Then a result of Grauert (\cite{Gra58}) implies that the fibration $S^\C/\Gamma\to S^\C/S^\C_\Gamma$ is holomorphically trivial, i.e. $S^\C/\Gamma\cong S^\C_\Gamma/\Gamma\times \C^k$. Thus, $S^\C/\Gamma$ is Stein if and only if so is $S^\C_\Gamma/\Gamma$.

Since $S_\Gamma$ is also split-solvable, from  Thm.~\ref{LOEB} it follows that $S^\C_\Gamma/\Gamma$ is Stein if and only if  $\ad(X)$ has no non-zero eigenvalues for all $X\in\s_\Gamma$.
 By Engel's theorem (see \cite[Thm.~3.5.4.]{Var84}), this is equivalent to say that $S_\Gamma$ is nilpotent. Since $S_\Gamma$ is the smallest (closed), connected, real subgroup of $S$ containing $\Gamma$, it follows that $S_\Gamma^\C/\Gamma$ is Stein is nilpotent if and only if $\Gamma$ is nilpotent, which implies the statement.
\end{proof}


\section{The crown domain.}\label{CROWN}
\noindent

Let $G/K$ be an irreducible, non-compact Riemannian symmetric space,
with $G$ a connected, non-compact, simple Lie group embedded in its universal complexification $G^\C$. Let $G^\C/K^\C$ be its Lie group
complexification.
By construction the left  action  of  $G$  on $G/K$
extends to a holomorphic action on $G^\C/K^\C$.
However, such an extended action turns out
not to be proper (cf. \cite{AkGi90}). In particular
the Riemannian metric on $G/K$ does not extend to a $G$-invariant
metric on $G^\C/K^\C$.
Then it is natural to look for $G$-invariant domains in $G^\C/K^\C$
on which the restriction of the $G$-action is proper.

In \cite{AkGi90} D. N. Akhiezer and S. G. Gindikin pointed out a
natural candidate, which turns out to be canonical from several
points of view. Let $\k\oplus \p$ be the Cartan decomposition of
$\g$ induced by $K$ and let $\Sigma=\Sigma(\a,\g)$ be the restricted root system
associated to a chosen maximal abelian subalgebra $\a$ of $\p$. Consider
the convex polyhedron
$$\omega:=\Big\{X\in\a:\big|\alpha(X)\big|<\frac \pi2, \;\text{ for every }\alpha\in\Sigma\Big\}.$$
\begin{defi} The \emph{crown domain} associated to the Riemannian symmetric space of
the non-compact type $G/K$ is defined by
$$\Xi:=G\exp(i\omega)\cdot p_0\;,$$
where $p_0:=eK^\C$ is the base point in $G^\C/K^\C$.
\end{defi}

 It is easy to check that $\Xi$ does not depend on the
choice of $\a$ and it was proved in \cite{AkGi90} that the $G$-action on
$\Xi$ is proper. For examples of crown domains, see the tables in the next section.

In \cite{AkGi90} it was also conjectured that the crown domain is
Stein, giving evidence of this fact in  several examples. The
conjecture was positively solved in \cite{BHH03} (cf. \cite{Bar03}, \cite{KrSt05}).
For this an important tool is the characterization of plurisubharmonic $G$-invariant functions on $\Xi$ as those function
whose restriction on  the $G$-slice $\exp(i\omega)\cdot p_0$
is  $W$-invariant and convex.
Here $W$ denotes the Weyl group with respect to the maximal abelian subalgebra $\a$ of $\p$.

We  are  going to use such a characterization in order to show that the crown domain is taut. Let us first recall the following definitions.

\begin{defi} A complex manifold is \emph{taut} if the
family of holomorphic discs $\O(\Delta,X)$ is normal.
That is, for a sequence of holomorphic discs
$\{f_j:\Delta \longrightarrow X\}_j$ there are two possibilities
\begin{itemize}
    \item[{\rm (1)}] admits a subsequence $\{f_{j_k}\}$ which converges
    uniformly on compact subsets to a holomorphic disc in $\O(\Delta,X)$, or
    \item[{\rm (2)}]  is compactly divergent, i.e. given two compact subsets $K\subset\Delta$ and $L\subset X$ there exists $\nu\in\N$ such
    that $f_j(K)\cap L=\varnothing$ for every $j>\nu$.
\end{itemize}
\end{defi}

\begin{defi} (cf. \cite{Ste75}) A Stein manifold  is
\emph{hyperconvex} if it admits a bounded,  continuous
plurisubharmonic exhaustion.
\end{defi}

\noindent Hyperconvex manifolds are taut  (see \cite[Cor.~5]{Sib81}) and taut manifolds are Kobayashi hyperbolic (see
\cite{Kob98}). Here we show

\begin{prop} \label{TAUT} Let $G/K$ be a Riemannian symmetric
space of the non-compact type. Then the associated crown domain is taut.
\end{prop}
\begin{proof} By Theorem B in  \cite{Bor63} there exists a
discrete, cocompact, torsion-free subgroup $\Gamma$ of $G$.
Since the $G$-action on $\Xi$ is proper, it follows that
$\Gamma$  acts freely on $\Xi$ and the canonical projection $\Xi \to
\Xi/ \Gamma$ is a covering map.

Consider the negative,  strictly convex, $W$-invariant exhaustion
of $\omega$ defined by
$$u(\xi)=\sum_{\alpha\in\Sigma} \bigg ( \alpha^2(\xi) -\Big(\frac
\pi2\Big)^2 \bigg ),
\qquad\text{for }\xi\in\omega.$$ Since one has an isomorphism of
orbit spaces $\Xi/G \cong \omega/W$ (see \cite[Prop.~8]{AkGi90}),
the function $u$ extends to a bounded, $G$-invariant function on
$\Xi$, also denoted by $u$. By \cite[Thm.~10]{BHH03} the function $u$ is strictly plurisubharmonic.

Also note that $u$ pushes down to a bounded, continuous, strictly
plurisubharmonic function $\tilde u$ of $\Xi /\Gamma$. Since
$G/\Gamma$ is compact, the preimage $\tilde u^{-1}(C)$ of any compact subset $C\subset(-\infty, 0)$ is compact in $\Xi/\Gamma$. In
particular $\Xi/\Gamma$ is hyperconvex and  \cite[Cor.~5]{Sib81}
implies that $\Xi /\Gamma$ is taut. Since a covering of a taut manifold is taut by \cite[Cor.~4]{ThHu93}, it follows that
 $\Xi$ is taut as well.
 \end{proof}

Now we  recall a result of S. Gindikin
and B. Krötz which will be used in the sequel in order to
realize the crown domain as a Stein invariant domain in the universal
complexification of a maximal, split-solvable subgroup of $G$.
For this let us consider an Iwasawa decomposition
$NAK$ of $G$. Let $N^\C$, $A^\C$, and $K^\C$
denote the complexifications of $N$, $A$ and $K$ in $G^\C$,
as in  Proposition \ref{N}.
One can show that $N^\C A^\C K^\C$ is a proper, Zariski open subset of $G^\C$ (see \cite{SiWo02}) and in general $A^\C\cap K^\C\neq\{e\}$.
Hence $N^\C A^\C K^\C$ is not a decomposition of $G^\C$.

However, by considering the $A$-invariant domain
$T_\omega$ of $A^\C$ defined by the cell $\omega$ of the crown domain, i.e.
$T_\omega:=A\exp(i\omega)$, one obtains a tubular neighborhood
$N^\C T_\omega K^\C$ of $G$, which can be regarded as a
local complexification of the Iwasawa decomposition of $G$. Let $H^\C$ be the complexification of a real Lie group $H$.
In the sequel we will refer to a \emph{tube} domain in $H^\C$ as
an $H$-invariant domain  of $H^\C$. One has

\begin{lemma} \label{MULT} The multiplication map
$$\phi: N^\C \times T_{\omega}\times K^\C\longrightarrow G^\C$$
defined by $\phi(n,a,k):=nak$, is an open analytic biholomorphism onto
its image $N^\C T_{\omega}K^\C\subset G^\C$. In particular,
the map $N^\C\times T_{\omega}\longrightarrow G^\C/K^\C$, given by
$(n,a)\mapsto na\cdot p_0$, with $p_0=eK^\C,$ defines an  equivariant
biholomorphism between a tube of $N^\C\rtimes A^\C$ and a Stein,
$N\!A$-invariant domain of $G^\C/K^\C$.
\end{lemma}
\begin{proof} Arguing as in \cite[Prop.~1.3]{KrSt04}, we first
show that the differential $d\phi$ is everywhere surjective. Since
$\phi$ is left $N^\C$ and right $K^\C$-equivariant, it is enough
to show that $d\phi_{(e,a,e)}$ is surjective for every $a\in
A^\C$. Identify $\n^\C\times\a^\C\times\k^\C$ with
$T_{(e,a,e)}(N^\C \times A^\C\times K^\C)$ via the linear
isomorphism $(X,Y,Z)\longmapsto (X,(L_a)_*(Y),Z)$. Then for
$X \in \n^\C, Y \in \a^\C$ and $Z \in \k^\C$ one has
\begin{description}
    \item[] $d\phi_{(e,a,e)}(X)=\frac
    d{ds}\bigg|_{s=0}\phi(\exp(sX),a,e)=(R_a)_*(X)$,
    \item[] $d\phi_{(e,a,e)}(Y)=\frac d{ds}\bigg|_{s=0}
    \phi(e,a\exp(sY),e)=(L_a)_*(Y)$,
    \item[] $d\phi_{(e,a,e)}(Z)=\frac d{ds}\bigg|_{s=0}
    \phi(e,a,\exp(sZ))=(L_a)_*(Z).$
\end{description}
Hence
$$d\phi_{(e,a,e)}(X,Y,Z)=(L_{a})_*(\adj(a^{-1})X+Y+Z)\,,$$
implying that $d\phi_{(e,a,e)}$ is surjective, as claimed.
\smallskip\\
It remains to check that $\phi$ is injective. Suppose that
$nak=n'a'k'$ with $n,n'\in N^\C$, $a,a'\in T_\omega$ and $k,k'\in
K^\C$. Let $\theta$ be the holomorphic extension to $G^\C$ of the
Cartan involution of $G$ with respect to $K$. Then
$nak\,\theta(nak)^{-1}=n'a'k'\theta(n'a'k')^{-1}$, and
consequently
$$na^2\theta(n^{-1})=n'(a')^2\theta((n')^{-1}).$$
Since by (iv) of Prop.~\ref{N} the map $N^\C\times
A^\C\times\theta(N^\C)\longrightarrow N^\C A^\C \theta(N^\C)$
given by $(n,a,\theta(n))\longmapsto na\theta(n)$ is injective,
it follows that $n=n'$ and $a^2=(a')^2$. As a consequence, for
 $a=t\exp(iX)$ and $a'=t\exp(iX')$, with $t,t'\in A$ and $X, X'\in\omega$,
 one has  $t=t'$ and $\exp(i2X)=\exp(i2X')$. Thus, in order to show that
$a=a'$ it is enough to check that $\exp|_{i2\omega}$ is injective.
By choosing a suitable basis of $\g^\C$ (cf. proof of Prop. 1.3 in
\cite{Vit12}), one sees that  $\ad_{\g^\C}$ maps $2i\omega$ into the diagonal matrices of the form
$\Big\{d(i\mu_1,\ldots,i\mu_n)\in\d(n,\C):\,|\mu_i|
<\pi \ {\rm and}\ \;\sum_i\mu_i=0\Big\}$. Then,
from the following commutative diagram,
 $$
\begin{array}{rcc}
  \g^\C\supset 2i\omega & \mathop{\makebox[1.5cm][c]{\rightarrowfill}}
  \limits^{\ad_{\g^\C}} & \hspace{-.5cm} \d(n,\C)\subset \mathfrak{gl}(n,\C)\smallskip\\
\exp_{|_{2i\omega}} \Bigg\downarrow & & \hspace{-1.3cm}\Bigg\downarrow e_{|_{\d(n,\C)}}\smallskip\\
 G^\C\supset \exp(2i\omega) & \mathop{\makebox[1.5cm][c]{\rightarrowfill}}
 \limits_{\adj_{G^\C}} &D(n,\C)\subset GL(n,\C)
\end{array}
$$
one sees that $\exp |_{2i\omega}$ is necessarily injective, as claimed.

For the last statement, note that the tube $T_\omega$ of $T^\C \cong \C^r$
has convex base, therefore it
is  Stein. As a consequence  $N^\C\times T_{\omega}$
is a Stein tube in $N^\C \rtimes A^\C$.
Finally, we proved above that the map
$$N^\C\times T_{\omega}\longrightarrow N^\C T_\omega\cdot p_0\,,$$
defined by $(n,a)\longmapsto na\cdot p_0$, is a biholomorphism.
Since the abelian subgroup $A$ normalizes $N$, it also
normalizes $N^\C$. Moreover, for $(n',a')$ in $N\rtimes A\cong NA$
one has $(n',a')(n,a)\cdot p_0= n'(a'n(a')^{-1})a'a \cdot p_0 =
n'a' \cdot (na\cdot p_0)$, implying that the map
is equivariant with respect to $N\rtimes A\cong NA$.
\end{proof}

\noindent By \cite[Thm.~1.3]{GiKr02a}, one has
\begin{teo} \label{Id} The crown domain $\Xi$ is the connected component containing $G/K$ of the intersection $\bigcap_{g\in G}gN^\C T_\omega\cdot p_0$.
\end{teo}

\noindent
Then, as a consequence of Lemma \ref{MULT} and the above theorem,
$\Xi$ is an $N\!A$-invariant domain of $N^\C T_\omega$,
therefore one has
\begin{cor}\label{STEINTUBE} The crown domain is biholomorphic to a Stein tube in $N^\C\rtimes A^\C$ contained in $N^\C\times T_\omega$.
\end{cor}

\medskip
In the next example we give an explicit realization of the tube $N^\C T_\omega\cdot p_0$ of $G^\C/K^\C$ which, by the above corollary, contains the crown domain.
\medskip

\begin{ex}
\label{EXAMPLE} Consider the maximal compact subgroup $K=SO(3)$
of  $G=SL(3,\R)$ and the associated Riemannian symmetric space of
the non-compact type $G/K$.
Let $\k\oplus\p$ be the Cartan's decomposition of the Lie algebra $\g=\mathfrak{sl}(3,\R)$, where
$\k=\mathfrak{so}(3)$ and $\p$ consists of all symmetric matrices of trace zero. Consider the maximal
abelian subspace $\a=\{d(\lambda_1,\lambda_2,\lambda_3)\in \d(3,\R)|\sum_k\lambda_k=0\}$ of $\p$, where $d(\lambda_1,\lambda_2,\lambda_3)$ denotes the real diagonal matrix whose entries are $\lambda_1,\lambda_2,\lambda_3$.
Let $\varepsilon_k:\a\rightarrow\R$ be defined by
$\varepsilon_k\big(d(\lambda_1,\lambda_2,\lambda_3)\big):=\lambda_k$,
for $k=1,2,3$,
and choose the set of positive roots
$\Sigma^+=\{\varepsilon_k-\varepsilon_h:\;1\leq k<h\leq 3\}$. The associated Iwasawa decomposition is given by $G=KAN$, where $K=SO(3)$,
$A=\{d(e^{\lambda_1},e^{\lambda_2},e^{\lambda_3}):\sum_k\lambda_k=0\}$
and $N=U(3,\R)$ consists of the real unipotent, upper
triangular matrices. The cell of the crown domain is given by
$$\omega=\Big\{d(\lambda_1,\lambda_2,\lambda_3)\in\a:\;
|\lambda_k-\lambda_h|<\pi/2;\;1\leq k,h\leq 3,\;\sum_k\lambda_k=0\Big\}
\atop \cong\Big\{(\lambda_2,\lambda_3)\in\R^2:\,|2\lambda_2+\lambda_3|<\pi/2,
|2\lambda_3+\lambda_2|<\pi/2,|\lambda_2-\lambda_3|<\pi/2\Big\}\, .$$

\begin{figure}[!hH]
 \centering
 \input{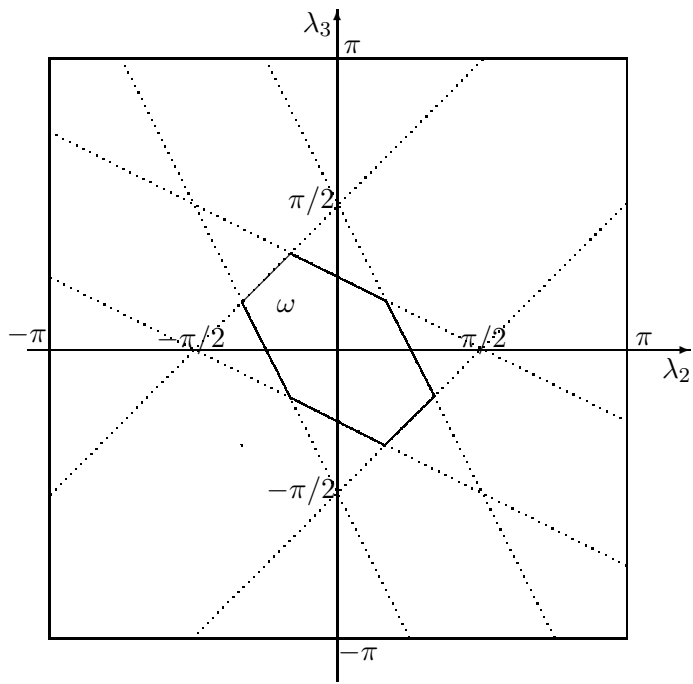}
  \caption{The cell $\omega$ of the crown domain associated to $SL(3,\R)/SO(3)$.}\label{CELL}
\end{figure}

\smallskip
Consider the $G^\C$-action on the complex $3\times 3$ matrices given by $g\cdot A:=\,^t\!gAg$. The $G^\C$-orbit
through the identity $I$ consists of the complex unimodular symmetric
$3\times 3$-matrices $SM(3,\C)$. Moreover, since the isotropy group at $I$ is $K^\C$ one has $G^\C/K^\C \cong SM(3,\C)$.

We are interested in the tube $N^\C T_\omega \cdot I\subset SM(3,\C)$.
For this let
$n:=\left(
   \begin{array}{ccc}
     1 & \alpha & \beta \\
     0 & 1 &  \gamma \\
     0 & 0 & 1 \\
   \end{array}
 \right)$,
with $\alpha,\beta,\gamma\in\C$, be an element of $N^\C$,
and let $a:=d(\zeta_1,\zeta_2,\zeta_3)$,
with $|\arg(\zeta_h\zeta_k^{-1})|<\frac\pi2$ and $\zeta_1\zeta_2\zeta_3=1$,
be an element of $T_\omega$.
A straightforward computation shows that
$$na\cdot I=
\left(
  \begin{array}{ccc}
    \zeta_1^2+\alpha^2\zeta_2^2+\beta^2\zeta_3^2 & \alpha\zeta^2_2+\gamma\beta\zeta_3^2 & \beta\zeta^2_3 \\
    \alpha\zeta^2_2+\gamma\beta\zeta_3^2 & \zeta^2_2+\gamma^2\zeta^2_3 & \gamma\zeta_3^2 \\
    \beta\zeta^2_3 & \gamma\zeta_3^2 & \zeta_3^2 \\
  \end{array}
\right)
$$
Then, for $na\cdot I=(a_{ij})_{ij}$ one has
\begin{equation}\label{BIHOLOMORPHISM}
\left\{
  \begin{array}{ll}
    \zeta_3^2=a_{33} \\
    \zeta_2^2=a_{33}^{-1}(a_{22}a_{33}-a_{23}^2)  \\
    \zeta_1^2=(a_{22}a_{33}-a_{23}^2)^{-1}  \\
    \gamma=a_{33}^{-1}a_{23}  \\
    \beta=a_{33}^{-1}a_{13} \\
    \alpha=(a_{12}a_{33}-a_{23}a_{13})(a_{22}a_{33}-a_{23}^2)^{-1}
  \end{array}
\right.
\end{equation}

\noindent Therefore the tube $N^\C T_\omega \cdot I$ is contained in the subdomain $E$ of $SM(3,\C)$ consisting of the elements $(a_{ij})_{ij}$ satisfying the following inequalities
\begin{equation}\label{E}
\left\{%
\begin{array}{ll}
    a_{33}\neq0,\;a_{22}a_{33}-a_{23}^2\neq0\medskip\\
    \big|\arg\big(a_{33}^{-1}(a_{22}a_{33}-a_{23}^2)^2\big)\big|<\pi\medskip\\
    \big|\arg(a_{33}(a_{22}a_{33}-a_{23}^2)\big)\big|<\pi\medskip\\
    \big|\arg\big(a_{33}^{-2}(a_{22}a_{33}-a_{23}^2\big)\big|<\pi\\
\end{array}%
\right.
\end{equation}

\noindent In fact $E$ coincides with $N^\C T_\omega\cdot I$. For this let $(a_{ij})_{ij}$ in $E$ and note that for all $(l,m)\in\Z^2\backslash\{(0,0)\}$ one has $\omega\cap(\omega+(l\pi,m\pi))=\emptyset$. Then the conditions in (\ref{E}) imply that there exist unique $(\lambda_2, \lambda_3)\in\omega$ such that $2\lambda_2=\arg(a_{33}^{-1}(a_{22}a_{33}-a_{23}^2))$ and $2\lambda_3=\arg(a_{33})$. Define $\zeta_2=\rho_2e^{i\lambda_2}$
and $\zeta_3=\rho_3e^{i\lambda_3}$, with $\rho_2^2=|\,a_{33}^{-1}(a_{22}a_{33}-a_{23}^2)|$ and $\rho_3^2=|\,a_{33}|$. Then, for $\zeta_3:=(\zeta_1\zeta_2)^{-1}$ and $\alpha, \beta, \gamma$ be defined by the equalities in (\ref{BIHOLOMORPHISM}),
 one has $na\cdot I=(a_{ij})_{ij}$, where $a:=d(\zeta_1,\zeta_2,\zeta_3)$ and
$n:=\left(
   \begin{array}{ccc}
     1 & \alpha & \beta \\
     0 & 1 &  \gamma\\
     0 & 0 & 1 \\
   \end{array}
 \right)$.
 Hence, $N^\C T_\omega\cdot I=E$.

\end{ex}

\section{The automorphism group of the crown domain.}\label{TABLES}
\noindent

Here we discuss the group $\Aut(\Xi)$ of biholomorphisms
of the crown domain. Since every isometry $f$ of $G/K$ extends to a biholomorphism of $\Xi$, it follows that the isometry group $\Iso(G/K)$ is contained in $\Aut(\Xi)$. By \cite[Thm.~6]{BHH03} only two cases arise:
either $\Xi$ is a Hermitian symmetric space of a larger group or $\Aut(\Xi)\cong\Iso(G/K)$.
In the latter case one says that $\Xi$ is \emph{rigid}.

If $G/K$ is Hermitian symmetric itself, then
 $\Xi$ is biholomorphic to the product
$G/K \times \overline {G/K}$, where $\overline {G/K}$ is the
Hermitian symmetric space with the opposite complex structure
(see  \cite[p.~5]{BHH03}, cf. \cite{FHW05}).
In this case $G/K$ is contained in $\Xi$ as the totally-real diagonal
$\{(gK, gK) \in G/K \times \overline {G/K}: g \in G\}$ and
the $G$-action on $\Xi$ is the diagonal one.
In particular $\Xi$ is a (reducible) Hermitian symmetric space.
There are also cases when
$\Xi$ turns out to be irreducible, as the following tables show
(cf. Table~1 in \cite{BHH03}).

\begin{table}[!hH]
\centering
\begin{tabular}{p{4.5cm}  l}
  \hlinewd{.7pt}\noalign{\medskip}
  \multicolumn{2}{c}{$G/K$ Hermitian symmetric \
  ($\Xi\cong G/K\times\overline{G/K}$)} \\[3pt]
  \hline\noalign{\medskip}
  \multicolumn{2}{l}{$\hspace{3cm} SU(p,q)/S(U_p\times U_q)$} \\[3pt]
  \multicolumn{2}{l}{$\hspace{3cm} SO_o(p,2)/(SO(p)\times SO(2))$} \\[3pt]
  \multicolumn{2}{l}{$\hspace{3cm} SO^*(2n)/U(n)$} \\[3pt]
  \multicolumn{2}{l}{$\hspace{3cm} Sp(n,\R)/U(n)$} \\[3pt]
  \multicolumn{2}{l}{$\hspace{3cm} (\e_{6(-14)},\mathfrak{so}(10)+\R$)} \\[3pt]
  \multicolumn{2}{l}{$\hspace{3cm} (\e_{7(-25)},\e_6+\R$)} \\[3pt]
  \hlinewd{.7pt}\noalign{\medskip}
  \multicolumn{2}{c}{$G/K$ not Hermitian symmetric with $\Xi$ Hermitian symmetric} \\[3pt]
  \hline\noalign{\smallskip}
  $SO_o(p,1)/SO(p),\ \ \  p>2$ & $\Xi\cong SO_o(p,2)/(SO(p)\times SO(2))$ \\[3pt]
  $Sp(p,q)/(Sp(p)\times Sp(q))$ & $\Xi\cong SU(2p,2q)/S(U_{2p}\times U_{2q})$ \\[3pt]
  $(\f_{4(-20)},\mathfrak{so}(9))$ & $\Xi\cong(\e_{6(-14)},\mathfrak{so}(10)+\R)$ \\[3pt]
  \hlinewd{.7pt}
\end{tabular}
\caption{Riemannian symmetric spaces of the non-compact type with
Hermitian symmetric crown domain.}
\label{HS}
\end{table}
\begin{table}[!hH]
\centering
\begin{tabular}{l}
  \hlinewd{.7pt}\noalign{\medskip}
  \multicolumn{1}{c}{G/K with $\Xi$ rigid} \\[3pt]
  \hline\noalign{\medskip}
  $SL(n,\R)/SO(n),\quad n>2$ \\[3pt]
  $SO_o(p,q)/(SO(p)\times SO(q)),\quad p,\;q>2$ \\[3pt]
  $SU^*(2n)/Sp(n)$ \\[3pt]
  $SL(n,\C)/SU(n),\quad n>2$ \\[3pt]
  $SO(n,\C)/SO(n),\quad n>3$\\[3pt]
  $Sp(n,\C)/Sp(n),\quad n>1$\\[3pt]
  All the other exceptional cases\\[3pt]
  \hlinewd{.7pt}
\end{tabular}
\caption{Riemannian symmetric spaces of the non-compact type with
rigid crown domain.}
\label{RIGID}
\end{table}

In the above tables there are some overlaps.
For the isomorphisms between irreducible Riemannian symmetric spaces
of low dimension we refer to \cite[p.~519]{Hel01}.


\section{Quotients of  crown domains by cyclic groups.}
\noindent

Let $X$ be a Stein manifold endowed with a proper action of a
discrete subgroup $\Gamma$ of $\Aut(X)$. In general one cannot expect the quotient $X/\Gamma$ to be Stein (cf. examples in \cite[Sect.~3]{Win90}, \cite{Oel92}, \cite[Sect.~5]{MiOe09}).
In some special cases it is possible to give necessary and/or sufficient
conditions so that $X/\Gamma$ has nice properties, e.g. it is Kähler
(see \cite{HuOe86}, \cite{Loe85}),
 Stein (cf. \cite{GiHu78},
\cite{Loe85}, \cite{HuOe86}, \cite{OeRi88}, \cite{MiOe09},
\cite{Mie10}). Of course if $\Gamma$ is finite, then the space
$X/\Gamma$ is Stein by \cite[Thm.~1,Ch.~V]{GrRe79}
(cf. \cite{Hei91}).

Here we consider the case  of an infinite cyclic
group $\Gamma \cong \Z$ acting properly by biholomorphisms
on a Stein domain $X$ of $\C^n$.
In this situation there exist positive results.
For instance, we already recalled  C. Miebach's result (\cite{Mie10}), stating that if $X$ is a bounded, homogeneous domain of $\C^n$ the quotient $X/\Z$ is Stein.
The quotient is also Stein when $X$ is a simply connected domain
of $\,\C^2$ and the $\Z$-action is induced by a proper $\R$-action
(see \cite{MiOe09}).
In fact, under the latter assumption, we are not aware of any
example of a domain $X$ of $\C^n$ such that
$X/\Z$ is not Stein for $n>2$ (cf. \cite[Sect.~5]{MiOe09}).

Recall that every  crown domain $\Xi$  is biholomorphic to a (taut, by Prop.~\ref{TAUT})
Stein domain of $\C^n$, since it can be regarded as a Stein tube in the universal covering
$\,N^\C \rtimes \widetilde {A^\C}$ of $\,N^\C \rtimes A^\C\,$ and
$\,N^\C \rtimes \widetilde {A^\C}$ is biholomorphic to $\,\C^l \times \C^r$
(see the proof of Thm.~\ref{MAIN} below and cf. Cor.~\ref{STEINTUBE}
for more details).
In this section we will show that $\Xi/\Z$ is Stein for any proper $\Z$-action.
For those crown domains which are Hermitian symmetric spaces of a larger group (see Table~\ref{HS}), this fact follows directly from C. Miebach's result. Indeed, Hermitian symmetric spaces
are biholomorphic to bounded symmetric domains
via the Harish-Chandra embedding (see \cite{Ha-Ch56}). For rigid crown domains (see Table~\ref{RIGID}) the result yields new
interesting examples of non-homogenous,
Stein, taut domains of $\C^n$ with a large group of automorphisms such that the quotient by a proper action of a cyclic group is Stein.

\subsection{Reduction to the case of automorphisms in a maximal
split-solvable subgroup of $G$.}\label{SPLITSOLVABLE}
\noindent

Let $G/K$ be an irreducible Riemannian symmetric space of the non-compact type
with rigid crown domain $\Xi$.
In order to prove that the quotient $\Xi/\Gamma$ is Stein for every proper
action of an infinite cyclic group $\Gamma \cong \Z$ of biholomorphisms, we apply the argument pointed out by
 C. Miebach in  \cite[Sect.~3]{Mie10}
   showing that one can reduce to the case
of $\Gamma$ contained in a maximal split-solvable subgroup of
$G$.
 For the sake of completeness
we carry out the details in our situation.

 First note that it is enough to consider the case of
$\Gamma$ lying in  the connected component  of the identity
in $\Iso(G/K)$,
which is given by  $G/Z$, where $Z$ is the (finite) center of $G$
(cf. \cite[Thm.~4.1, Ch.~V]{Hel01}).
By Remark~\ref{KLEIN} the center $Z$ plays no role in
the geometry of $\Xi$, so we will assume that it is trivial.
Since $G/K$ is a Riemannian manifold, every isotropy subgroup
of $\Iso(G/K)$ is compact  (\cite[Thm.~2.4, Ch.~IV]{Hel01}). In particular
all isotropy subgroups have a finite number of connected components.
 As a consequence so does $\Iso(G/K)$,
being $G/K$  homogeneous with respect to the connected component
of the identity in $\Iso(G/K)$.
Thus $\Gamma/(\Gamma\cap G)$ is finite and consequently
$\Xi/\Gamma$ is Stein if and only if so is $\Xi/(\Gamma\cap G)$.

So let $\Gamma$ be contained in the connected
component $G$ of the neutral element  in $\Iso(G/K)$.
Consider an Iwasawa decomposition $N\!AK$ of $G$.
We want to reduce to the case
of $\Gamma$ lying in the split-solvable subgroup $N\!A$ of $G$
(cf. Rem.~\ref{NA}). For this let
$\gamma\in G$ be a generator of $\Gamma$ and consider its
Jordan-Chevalley decomposition $\gamma=\gamma_u\gamma_h\gamma_e$, where
$\gamma_u$ is unipotent, $\gamma_h$ is hyperbolic and $\gamma_e$ is elliptic
(cf. Prop.~\ref{JORDAN}). Set $\gamma':=\gamma_u\gamma_h$ and
let $\Gamma':=\langle\gamma'\rangle$ be the subgroup generated by
$\gamma'$. By $(iv)$ of Proposition~\ref{jc},
we may assume that $\Gamma'$ is contained
in $N\!A$. Moreover the exponential map of $N\!A$ is a diffeomorphism
(see Rem.~\ref{SPLIT}), implying that  $\Gamma'$
is  closed
in $N\!A$, and consequently in $G$.
Thus it acts properly and freely on the domain $\Xi$.
In particular the canonical projections $p:\Xi \to \Xi/\Gamma$
and $p':\Xi \to \Xi/\Gamma'$ are coverings.

Since $\gamma_e$ is conjugated to an element of $K$, the closure $T$
of the cyclic subgroup generated by $\gamma_e$ is a compact abelian
subgroup of $G$.  Note that  $\gamma_e$ and $\gamma'$ commute,
therefore   $T\Gamma'$ is a subgroup of $G$.
In fact, since  all the elements in the Jordan-Chevalley decomposition
 of $\Gamma$ commute, one has $T \times \Gamma \cong T\Gamma=T\Gamma' \cong T \times \Gamma'$.  Therefore the compact group
$T$ acts (properly) on both  $\Xi/\Gamma$
and $\Xi/\Gamma'$ and one has
 a commutative diagram

\begin{displaymath}
\xymatrix@R=1em{
           & \Xi \ar[rd]^{p'} \ar[ld]_p      &   \\
 \Xi/\Gamma  \ar[rd]_q &                     &  \Xi/\Gamma' \ar[ld]^{q'}\\
           & Y                         &         }
\end{displaymath}
where $Y:= (\Xi/\Gamma)/T= (\Xi/\Gamma')/T$ and the canonical projections $q$ and $q'$ are proper. One has (\cite[Prop. 3.6]{Mie10})

\begin{prop} Let $\Gamma$ be a cyclic discrete group acting properly by biholomorphisms on $\Xi$ and let $\Gamma'$ be the above defined infinite cyclic subgroup of $NA$. The complex manifold $ \Xi/\Gamma$ is Stein if and only if so is $ \Xi/\Gamma'$.
\end{prop}
\begin{proof}
The proof is carried out by noting that given a smooth,
$T$-invariant function $f$ of $\Xi/\Gamma'$, there exists a unique smooth,
$T$-invariant function $f'$
on $\Xi/\Gamma$ such that $f(p(z))= f'(p'(z))$ for all $z$ in $\Xi$.
Since $p$ and $p'$ are coverings, the function $f'$ is strictly
plurisubharmonic if and only if so is $f$. Moreover, by using the commutativity of the above diagram and
the properness of the projections $q$ and $q'$, one checks that $f$ is an exhaustion if and only
if so is $f'$.
As a consequence $\Xi/\Gamma$ admits a smooth, $T$-invariant, strictly plurisubharmonic exhaustion if and only if so does $\Xi/\Gamma'$.
Since by Thm.~\cite[Thm.~5.2.10]{Hor66} and by integration over $T$
one has that  $\Xi/\Gamma$ is Stein if and only if it admits a smooth,
$T$-invariant, strictly plurisubharmonic exhaustion, this implies the statement.
For further  details we refer to the original proof
in \cite[Prop.~3.6]{Mie10}.
\end{proof}

\subsection{The main result.}

We first need the following lemma.

\begin{lemma}\label{D/Gamma} Let $\Gamma$ be a discrete group which acts freely and properly discontinuously on a Stein manifold $X$ and let $D$ be a $\Gamma$-invariant, Stein subdomain of $X$. If $X/\Gamma$ is Stein, then so is $D/\Gamma$.
\end{lemma}
\begin{proof}
First recall that by a classical result of
F. Docquier and H. Grauert (\cite{DoGr60}) a domain $O$ in a Stein
manifold $Z$ is Stein if and only if it is locally Stein, i.e.
for every element $z$ of the boundary of $O$ there exists a neighborhood
$V$ of $z$ in $Z$ such that $V\cap O$ is Stein.
Hence, it is enough to prove that $D/\Gamma$ is locally Stein in $X/\Gamma$.

For this note that the canonical projection $\pi:X\longrightarrow X/\Gamma$ is a covering map and for $z\in \partial(D/\Gamma)$ choose $x\in\pi^{-1}(z)$.
Since  $D$ is $\Gamma$-invariant, $\pi^{-1}(\overline{D/\Gamma})=\overline{D}$, therefore $x\in\partial D$.
Let $U$ be a Stein neighborhood of $x$ such that $gU\cap U=\emptyset$
for every $g\in\Gamma\setminus\{e\}$.
Then $\pi:U\longrightarrow\pi(U)$ is a biholomorphism and $\pi(U)\cap (D/\Gamma)=\pi(U\cap D)$ is Stein.
Therefore $D/\Gamma$ is locally Stein in $X/\Gamma$, implying the
statement.
\end{proof}

\noindent Our main result is

\begin{teo}\label{MAIN}
Let $G/K$ be an irreducible Riemannian symmetric
space of the non-compact type and
let $\Xi$ be the associated  crown domain.
Then $\Xi/\Z$ is a Stein manifold  for every proper $\Z$-action .
\end{teo}

\begin{proof} If the crown domain $\Xi$ is not rigid then
it is biholomorphic to a bounded symmetric domain and the statement follows from C. Miebach's result (\cite{Mie10}).
So consider the case of $\Xi$ rigid.
Let $NAK$ be an Iwasawa decomposition of
$G$. By the argument in Section~\ref{SPLITSOLVABLE},
one may assume that $\Z$ is contained in the
split-solvable subgroup $N\!A \cong N \rtimes A$ of $G$.
Consider the universal coverings
$\widetilde {A^\C}$ of $A^\C$ and
$N^\C \rtimes
\widetilde {A^\C}$ of $N^\C \rtimes A^\C$, respectively.
Since the Lie group $N \rtimes A$ is simply connected, it lifts to a real form
of $N^\C \rtimes \widetilde {A^\C}$
(in fact $N^\C \rtimes \widetilde {A^\C}$
is the universal complexification of $N \rtimes A$).
Moreover, the crown domain $\Xi$ is simply connected
and by Corollary~\ref{STEINTUBE} it is biholomorphic to a Stein,
$N \rtimes A$-invariant
domain of $N^\C \rtimes A^\C$.
Thus it lifts to a Stein tube in $N^\C \rtimes \widetilde {A^\C}$.
One has a commutative diagram
$$
\xymatrix{
  \Xi\;\ar[d] \ar@{^{(}->}[r] & N^\C\rtimes \widetilde{A^\C}  \ar[d] \\
  \Xi/\Z\; \ar@{^{(}->}[r] & (N^\C\rtimes \widetilde{A^\C})/\Z\quad,}
$$
where the horizontal arrows are the natural inclusions and the
vertical arrows are the canonical covering maps.
By Remark~\ref{NA} the group
$N \rtimes A$ is split-solvable and from
Proposition~\ref{S/Z QUOTIENT} it follows that
$(N^\C\rtimes \widetilde{A^\C})/\Z$ is Stein.
Finally, being the crown domain $\Xi$ a Stein tube in $N^\C\rtimes \widetilde{A^\C}$, Lemma~\ref{D/Gamma} applies to show that
$\Xi/\Z$ is Stein, as wished.
\end{proof}

As suggested to us by C. Miebach,  a similar
argument as in the above proof applies to the case of
discrete nilpotent subgroups of $N\!A$.

\begin{prop}\label{SUGG}
Let $\Gamma$ be any discrete, nilpotent subgroup of $N\!A$.
Then $\Xi/\Gamma$ is Stein.
\end{prop}
\begin{proof}
By Proposition~\ref{S/Z QUOTIENT}, the quotient  $(N^\C\times \widetilde{A^\C})/\Gamma$ is a Stein manifold. Then Lemma~\ref{D/Gamma} implies that $\Xi/\Gamma$ is Stein, as wished.
\end{proof}

\begin{small}

\bigskip\noindent {\sc Dipartimento di Matematica, Università di Roma 'Tor Vergata', via della Ricerca Scientifica 1, 00133 Roma, Italy}.\\
\emph{E-mail address}: vitali@mat.uniroma2.it

\end{small}

\end{document}